\documentclass[12pt,twoside,leqno]{article}
\usepackage{amsthm,amsfonts,amssymb,amsmath, amscd, enumerate}
\usepackage[colorlinks=true]{hyperref}
\usepackage[mathscr]{eucal}
\usepackage[all]{xy}
\usepackage{mathrsfs}
\theoremstyle{plain}

\newcommand{\Addresses}{{
  \bigskip
  \footnotesize

  \textsc{Fakultät für Mathematik, Universität Regensburg, 93040 Regensburg, Germany}\par\nopagebreak
  \textit{E-mail address}: \texttt{yanshuai.qin@ur.de}

}}

\newcommand{\CO}{{\mathcal {O}}}

\newcommand{\CV}{{\mathcal {V}}}

\newcommand{\CX}{{\mathcal {X}}}
\newcommand{\CY}{{\mathcal {Y}}}

\newcommand{\Char}{{\mathrm{char}}}
\newcommand{\Br}{{\mathrm{Br}}}

\newcommand{\red}{{\mathrm{red}}}

\newcommand{\Gal}{{\mathrm{Gal}}}

\newcommand{\Hom}{{\mathrm{Hom}}}

\newcommand{\Ker}{{\mathrm{Ker}}}

\newcommand{\NS}{{\mathrm{NS}}}

\newcommand{\non}{{\mathrm{non}\text{-}}}

\newcommand{\Pic}{\mathrm{Pic}}

\font\cyr=wncyr10  \newcommand{\Sha}{\hbox{\cyr X}}

\newcommand{\tor}{{\mathrm{tor}}}

\DeclareMathOperator{\Spec}{Spec}

\newcommand{\QQ}{\mathbb{Q}}

\newcommand{\ZZ}{\mathbb{Z}} 
\newcommand{\FF}{\mathbb{F}} 
\newcommand{\GG}{\mathbb{G}}

\newcommand{\lra}{\longrightarrow}
\newcommand{\ra}{\rightarrow}

\newcommand{\et}{\mathrm{et}}

\newcommand{\SO}{{\mathscr{O}}}

\newcommand{\SF}{\mathscr{F}} 
 
\newcommand{\SA}{{\mathscr{A}}}
\newcommand{\SH}{{\mathscr{H}}}

\newtheorem{thm}{Theorem}[section]
\newtheorem{cor}[thm]{Corollary}
\newtheorem{lem}[thm]{Lemma}
\newtheorem{prop}[thm]{Proposition}
\newtheorem{conj}[thm]{Conjecture}

\theoremstyle{definition}

\theoremstyle{remark}
\newtheorem{rem}[thm]{Remark}

\begin{document}

\renewcommand{\thefootnote}{\fnsymbol{footnote}} 
\footnotetext{\emph{MSC classes}: 11G40, 14G17, 14J20, 11G25, 11G35}
\footnotetext{ The author is supported by the DFG through CRC1085 Higher Invariants (University of Regensburg).}
\title{On the Brauer groups of fibrations}
\author{Yanshuai Qin}

\maketitle

\begin{abstract}
Let $\CX\rightarrow C$ be a flat $k$-morphism between smooth integral varieties over a finitely generated field $k$ such that the generic fiber $X$ is smooth, projective and geometrically connected. Assuming that $C$ is a curve with function field $K$, we build a relation between the Tate-Shafarevich group of $\Pic^0_{X/K}$ and the geometric Brauer groups of $\CX$ and $X$, generalizing a theorem of Artin and Grothendieck for fibered surfaces to higher relative dimensions.
\end{abstract}

\section{Introduction}
For any regular noetherian scheme $X$,  the \emph{cohomological Brauer group} $\Br(X)$ is defined to be the \'etale cohomology group $H^2_{\et}(X,\GG_m)$.

Let $C$ be the spectrum of the ring of integers in a number field or a smooth projective geometrically connected curve over a finite field with function field $K$. Let $\mathcal{X}$ be a $2$-dimensional regular scheme and $\pi:\mathcal{X}\longrightarrow C$ be a proper flat morphism such that the generic fiber $X$ of $\pi$ is smooth and geometrically connected over $K$.  Artin and Grothendieck (cf.\cite[$\S$ 4]{Gro3} or \cite[Prop. 5.3]{Ulm}) proved that there is an isomorphism 
$$
\Sha(\Pic^{0}_{X/K})\cong \Br(\mathcal{X})
$$
up to finite groups. Assuming the finiteness of $\Sha(\Pic^{0}_{X/K})$ or $\Br(\mathcal{X})$,  Gordon \cite{Gor}, Milne \cite{Mil1}, Gonzales-Aviles \cite{Goa}, Liu-Lorenzini-Raynaud \cite{LLR1, LLR2}, Geisser \cite{Gei1} and Lichtenbaum-Ramachandran-Suzuki \cite{LRS} gave a precise relation between their orders.

In this paper, we will generalize the above result of Artin and Grothendieck to fibrations of higher relative dimensions over arbitrary finitely generated fields.  As an application, we give a simpler proof of Geisser's Theorem 1.1 in \cite{Gei2} and reprove the reduction theorem of the Tate conjecture to prime fields due to Andr\'e \cite{And2} and Ambrosi \cite{Amb}.
\subsection{Main results}
For any field $k$, denote by $k^s$ the separable closure. Denote by $G_k=\Gal(k^s/k)$ the absolute Galois group of $k$.
For a torsion abelian group $M$ and a prime number $p$, let $M(\non p)$ denote the subgroup of elements of order prime to $p$. We also define $M(\non p):=M$ for $p=0$. For two abelian groups $M$ and $N$, we say that they are \emph{almost isomorphic} if there exists a subquotient $M_1/M_0$ of $M$ (resp. $N_1/N_0$ of $N$) such that $M_0$ (resp. $N_0$) and $M/M_1$ (resp. $N/N_1$) are finite and $M_1/M_0\cong N_1/N_0$. Let $f:M\rightarrow N$ be a homomorphism with a finite cokernel and a kernel almost isomorphic to an abelian group $H$, in this situation, we say that the sequence $0\rightarrow H\rightarrow M\rightarrow N\rightarrow 0$ is exact up to finite groups.
\begin{thm}
\label{thm1.2}
Let $k$ be a finitely generated field of characteristic $p\geq 0$. Let $\mathcal{X}$ be a smooth geometrically connected variety over $k$ and $C$ be a smooth projective geometrically connected curve over $k$ with function field $K$. Let $\pi:\mathcal{X}\longrightarrow C$ be a flat $k$-morphism such that the generic fiber $X$ is smooth projective geometrically connected over $K$. Let $K^{\prime}$ denote $Kk^s$ and write $P$ for $\Pic^{0}_{X/K,\red}$. Define
$$\Sha_{K^\prime}(P):=\Ker(H^1(K^\prime, P)
\lra \prod_{v\in |C_{k^s}|}H^1(K_{v}^{sh}, P)),$$
where $|C_{k^s}|$ denotes the set of closed points in $C_{k^s}$ and $K_v^{sh}$ denotes the fraction field of a strictly henselian local ring of $C_{k^s}$ at $v$.
Then there is an exact sequence up to finite groups
$$
 0 \rightarrow \Sha_{K^\prime}(P)^{G_k}(\non p)\rightarrow \Br(\mathcal{X}_{k^s})^{G_k}(\non p)\rightarrow \Br(X_{K^s})^{G_K}(\non p)\rightarrow 0.
$$

\end{thm}
In the case that $k$ is a finite field of characteristic $p$ and $\CX$ is projective over $k$, the above result gives a generalization (up to $p$-torsion) of Artin-Grothendieck's theorem to higher relative dimensions in positive characteristics:
\begin{cor}\label{mainthm}
Let $\pi:\mathcal{X}\longrightarrow C$ be a proper flat $k$-morphism, where $C$ is a smooth projective geometrically connected curve over a finite field $k$ of characteristic $p$ with function field $K$. Assuming that $\mathcal{X}$ is smooth projective over $k$ and the generic fiber $X$ of $\pi$ is smooth projective geometrically connected over $K$, then there is an exact sequence up to finite groups
$$
0\rightarrow \Sha(\Pic_{X/K,\red}^0)(\non p)\rightarrow \Br(\mathcal{X})(\non p)\rightarrow \Br(X_{K^s})^{G_K}(\non p)\rightarrow 0.
$$
\end{cor}
In Theorem \ref{thm1.2}, when $k$ is finite, the canonical maps
 $\Sha(P)(\non p) \rightarrow \Sha_{K^\prime}(P)^{G_k}(\non p)$ and $\Br(\CX)\rightarrow \Br(\mathcal{X}_{k^s})^{G_k}$ have finite kernels and cokernels (cf. Proposition \ref{agsha} and \cite[Cor. 1.4]{Yua2}). Thus, Corollary \ref{mainthm} follows from Theorem \ref{thm1.2}. 
\begin{rem}
Two natural questions not pursued in our Theorem \ref{thm1.2} and Corollary \ref{mainthm} are
\begin{itemize}
    \item[(i)] the relation between $p$-primary part of those groups in characteristic $p>0$ and
\item[(ii)] formulas which express the discrepancy in the exact sequence.
\end{itemize}
For the first question, one may ask if the $p$-primary torsion of $\Br(X_{K^s})^{G_K}$ is still the obstruction of the Tate conjecture for divisors. D’Addezio\cite{DAd} proved that for an abelian variety $A/K$, the $p$-primary torsion of $\Br(A_{K^s})^{G_K}$ is of finite exponent. To address the $p$-torsion, one must delve into the flat cohomology of $X$ over an imperfect field. Establishing connections between flat cohomology over an imperfect field and crystalline or rigid cohomology poses a challenge. Thus, our existing approach cannot be extended to encompass the $p$-torsion scenario. For the second question, one may ask for an explicit bound for the cokernel of $\Br(X)\rightarrow \Br(X_{K^s})^{G_K}$. In this direction,  Colliot-Thél\`ene and Skorobogatov \cite[Thm. 2.2]{CTS1} gave interesting bounds for the cokernel of $\Br(X)\rightarrow \Br(X_{K^s})^{G_K}$ for some specific varieties over number fields.
\end{rem}

\subsection{Applications}
Let us first recall the Tate conjecture for divisors over finitely generated fields.
\begin{conj}\label{Tateconj} ( Conjecture $T^1(X,\ell)$)
Let $X$ be a projective and smooth variety over a finitely generated field $k$ of characteristic $p\geq0$, and $\ell\neq p$ be a prime number. Then the cycle class map\\
$$
\Pic(X)\otimes_\mathbb{Z}\mathbb{Q}_\ell\longrightarrow H^2_{\et}(X_{k^s},\mathbb{Q}_\ell(1))^{G_k}
$$
is surjective.
\end{conj}
The conjecture is proved for abelian varieties by Tate \cite{Tat3} (over finite fields), Zarhin \cite{Zar1, Zar2} (over positive characteristics) and Faltings \cite{Fal1, Fal2} (over characteristic zero), for K3 surfaces over characteristic zero by Andr\'e \cite{And1} and Tankeev \cite{Tan4, Tan5}, for $K3$ surfaces over positive characteristics by Nygaard \cite{Nyg}, Nygaard-Ogus \cite{NO}, Artin-Swinnerton-Dyer \cite{ASD}, Maulik \cite{Mau}, Charles \cite{Cha} and Madapusi-Pera \cite{MP}.

By a well-known result (cf. \cite[Prop. 2.5]{SZ1} or \cite[Thm. 1.2]{Yua2} ), $T^1(X,\ell)$ is equivalent to the finiteness of $\Br(X_{k^s})^{G_k}(\ell)$. Thus, Corollary \ref{mainthm} implies the following Theorem of Geisser\cite{Gei2}:
 \begin{cor}\label{Gei}
Notations as in Corollary \ref{mainthm}. Let $\ell\neq p$ be a prime, then 
$T^{1}(\mathcal{X},\ell)$ is equivalent to the finiteness of  $\Sha(\Pic^{0}_{X/K,\red})(\ell)$ and $T^1(X,\ell)$.
\end{cor}
\begin{rem}
Geisser \cite{Gei2} first studied fibrations of higher relative dimensions over finite fields by using \'etale motivic cohomology theory. He proved Corollary \ref{Gei} (cf. \cite[Thm. 1.1]{Gei2}) and gave a precise relation between the orders of the Tate-Shafarevich group of the Albanese variety of the generic fiber and the Brauer group of $\CX$ (cf. \cite[Thm. 1.2]{Gei2}) under the assumption of the finiteness of $\Br(\CX)$. For fibrations over the spectrum of the ring of integers in a number field of relative dimension $>1$, the above question was first studied by Tankeev (cf. \cite{Tan1,Tan2,Tan3}). In \cite{Qin}, we will prove a version of Corollary \ref{mainthm} over the spectrum of the ring of integers in a number field.
\end{rem}
Since Brauer groups are the obstruction of the Tate conjecture for divisors, in combination with a spreading out argument, Theorem\ref{thm1.2} will imply the following reduction theorem of the Tate conjecture due to Andr\'e \cite{And2} and Ambrosi \cite{Amb}:
\begin{thm}\label{reduction}
Let $k$ be a prime field and $K$ be a finitely generated field over $k$. Let $\ell$ be a prime different from char$(k)$. Assuming that $T^1(\mathcal{X},\ell)$ is true for all smooth projective varieties $\mathcal{X}$ over $k$, then $T^1(X,\ell)$ holds for all smooth projective varieties $X$ over $K$.
\end{thm}
\begin{rem}
$T^1(X,\ell)$ for smooth projective varieties over finite fields has been reduced to $T^1(X,\ell)$ for smooth projective surfaces $X$ over finite fields with $H^1(X,\CO_X)=0$ by the work of de Jong \cite{deJ2}, Morrow \cite[Thm. 4.3]{Mor} and Yuan \cite[Thm. 1.6]{Yua1}. In \cite{Kah},  Kahn proved that Tate conjecture in codimension 1 over a finitely generated field follows from the same conjecture for surfaces over its prime subfield.
\end{rem}

 In \cite{SZ1}, Skorobogatov and Zarhin conjectured that $\Br(X_{k^s})^{G_k}$ is finite for any smooth projective variety $X$ over a finitely generated field $k$. In \cite{SZ1,SZ2}, they proved the finiteness of $\Br(X_{k^s})^{G_k}(\non p)$ for abelian varieties and $K3$ surfaces($\Char(k)\neq 2$). By our Theorem \ref{thm1.2}, we reduce the question about the finiteness of $\Br(X_{k^s})^{G_k}(\non p)$ to the question over prime fields:
\begin{thm}\label{finitebrauer}
Let $k$ be a prime field and $K$ be a field finitely generated over $k$. Then the following statements are true.
\begin{itemize}
\item[(1)] If $k=\QQ$ and $\Br(X_{k^s})^{G_k}$ is finite for any smooth projective variety $X$ over $k$, then $\Br(X_{K^s})^{G_K}$ is finite for any smooth projective variety $X$ over $K$.
\item[(2)] If $k=\FF_p$ and the Tate conjecture for divisors holds for any smooth projective variety over $k$, then $\Br(X_{K^s})^{G_K}(\non p)$ is finite for any smooth projective variety $X$ over $K$.
\end{itemize}
\end{thm}

\subsection{Sketch of the proof of Theorem \ref{thm1.2}}
For the proof of Theorem \ref{thm1.2}, the left exactness will follow from Grothendieck's argument in \cite[$\S$ 4]{Gro3}. The tricky part is to show that the natural map
$$\Br(\CX_{k^s})^{G_k}(\non p)\lra \Br(X_{K^s})^{G_K}(\non p)$$
has a finite cokernel. The key idea is to use a pull-back trick developed by Colliot-Thél\`ene and Skorobogatov (cf. \cite{CTS1} or \cite{Yua2} ) to reduce it to cases of relative dimension $1$. Choose an open dense subset $U\subseteq C$ such that $\pi^{-1}(U)\lra U$ is smooth and proper. Write $\CV$ for $\pi^{-1}(U)$. By purity for Brauer groups, the natural map
$$\Br(\CX_{k^s})^{G_k}(\non p)\lra \Br(\CV_{k^s})^{G_k}(\non p)$$
is injective and has a finite cokernel. It suffices to show that
$$\Br(\CV)(\non p)\lra \Br(X_{K^s})^{G_K}(\non p)$$
has a finite cokernel. For simplicity, we assume that $X(K)$ is not empty. By the spectral sequence
$$ E^{p,q}_2=H^p(U,R^q\pi_*\GG_m)\Rightarrow H^{p+q}(\CV,\GG_m),$$
we get an exact sequence
$$ H^2(\CV,\GG_m)\lra H^0(U,R^2\pi_*\GG_m)\stackrel{d_2^{0,2}}{\lra} H^2(U,R^1\pi_*\GG_m).$$
And we will show there is a canonical isomorphism
 $$H^0(U,R^2\pi_*\GG_m)(\non p)\cong\Br(X_{K^s})^{G_K}(\non p).$$
So the question is reduced to show the vanishing of $d^{0,2}_{2}$. In the case of relative dimension $1$, this is obvious since $R^2\pi_*\GG_m=0
 $ by Artin's theorem \cite[Cor. 3.2]{Gro3}. In general, taking a smooth projective curve $Y\subseteq X$, then spreading out, we get a relative curve $\pi^\prime:\CY\lra C$  and a $C$-morphism $\CY\lra \CX$. This gives a commutative diagram
\begin{displaymath}
\xymatrix{ 
H^0(U,R^2\pi_*\GG_m)\ar[r] \ar[d]& H^2(U,R^1\pi_*\GG_m) \ar[d] 
\\
H^0(U,R^2\pi^\prime_*\GG_m) \ar[r] & H^2(U,R^1\pi^\prime_*\GG_m)
}
\end{displaymath}
We hope that the second column is injective. In fact, we need finitely many $\pi^{i}:\CY_i\lra C$ such that the induced morphism of abelian sheaves
$$R^1\pi_*\GG_m\lra \bigoplus_i R^1\pi_*^{i}\GG_m$$
is split up to abelian sheaves of finite exponent. The splitting
implies that the natural map
$$H^2(U,R^1\pi_*\GG_m)\lra \bigoplus_i H^2(U,R^1\pi^i_*\GG_m)$$
has a kernel of finite exponent. It follows that $d^{0,2}_2$ has an image of finite exponent. This shows that the cokernel of the natural map at the beginning is of finite exponent. Since the groups are cofinitely generated (cf. \S 2.3), this property of finite exponent translates to finite cokernel.

To get this splitness, we will choose $Y_i$ as the complete intersection of very ample divisors on $X$ such that there exists an abelian variety $A/K$ and a $G_K$-equivariant map
$$\Pic(X_{K^s})\times A(K^s)\lra \bigoplus_i\Pic(Y_{i,K^s})$$
with finite kernel and cokernel (cf. \cite[\S 3.3, p17]{Yua2}). 

\subsection*{Notations}
By a \emph{finitely generated field}, we mean a field which is finitely generated over a prime field.
By a \emph{variety} over a field $k$, we mean a scheme which is separated and of finite type over $k$. A morphism between varieties over $k$ is a $k$-morphism. For a smooth proper geometrically connected variety $X$ over a field $k$, we use $\Pic^0_{X/k}$ to denote the identity component of the Picard scheme $\Pic_{X/k}$. Denote by $\Pic^0_{X/k,\red}$ the underlying reduced closed subscheme of $\Pic^0_{X/k}$. The default sheaves and cohomology over schemes are with respect to the small \'etale site. So $H^i$ is the abbreviation of $H_{\et}^i$. For any abelian group $M$, integer $m$ and prime $\ell$, we set\\
$$M[m]=\{x\in M| mx=0\},\quad M_{\tor}=\bigcup\limits_{m\geq 1}M[m],\quad  M(\ell)=\bigcup\limits_{n\geq 1}M[\ell^n].$$
\section{Preliminary results}
\subsection{The Galois invariant parts of geometric Brauer groups}
In this section, we will show that the Galois invariant part of the geometric Brauer group of a smooth variety is a birational invariant up to finite groups. This will allow us to shrink the base of a fibration without changing the question.

The following results in the case of characteristic $0$ were proved in \cite{CTS1} (cf. \cite[Prop. 6.1 and Thm. 6.2 (iii)]{CTS1}).
\begin{prop}\label{purity}
Let $X$ be a smooth geometrically connected variety over a finitely generated field $k$ of characteristic $p\geq 0$. Let $U\subseteq X$ be an open dense subset. Then the natural map
$$ \Br(X_{k^s})^{G_k}(\non p)\lra \Br(U_{k^s})^{G_k}(\non p)$$
is injectie and has a finite cokernel.
\end{prop}
\begin{proof}
The injectivity follows from that $X_{k^s}$ is regular and irreducible. To show that the cokernel is finite, we need the lemma below. Let $Y$ denote $X-U$ with reduced scheme structure. By purity for Brauer group (cf. \cite[Thm. 1.1]{Ces}), removing a close subset of codimension $\geq 2$ will not change the Brauer group. Thus, by shrinking $X$, we may assume that $Y$ is regular and of codimension $1$ in $X$.  Let $\ell \neq p$ be a prime. By purity for Brauer group (cf. \cite[\S 6]{Gro3} and \cite{Fuj}), there is a canonical exact sequence
$$0\lra \Br(X_{k^s})(\ell)\lra \Br(U_{k^s})(\ell)\lra H^1(Y_{k^s},\QQ_{\ell}/\ZZ_{\ell}).
$$
Thus, it suffices to show that the group
$$H^1(Y_{k^s},\QQ_{\ell}/\ZZ_\ell)^{G_k}$$
is finite and vanishes for all but finitely many $\ell$. This follows from the lemma below.
\end{proof} 
The following lemma in characteristic $0$ was proved by Colliot-Thél\`ene and Skorobogatov (cf. \cite[Prop. 6.1]{CTS1}). For completeness, we present the proof for the cases in arbitrary characteristics, following their arguments.
\begin{lem}
Let $U$ be a regular variety over a finitely generated field $k$. Let $\ell\neq \Char(k)$ be a prime. Then 
$$H^1(U_{k^s},\QQ_{\ell}/\ZZ_\ell)^{G_k}$$
is finite and vanishes for all but finitely many $\ell$.
\end{lem}
\begin{proof}
Firstly, we will show that it suffices to prove the claim for smooth varieties over $k$. Notice that we can replace $k$ by a finite separable extension. Thus, we may assume that $U_{k^s}$ is irreducible. Let $V$ be an open dense subset of $U_{k^s}$. By semi-purity (cf. \cite[\S 8]{Fuj}),
$$H^1(U_{k^s},\ZZ/\ell^n)\lra H^1(V,\ZZ/\ell^n)$$
is injective. We may assume that $V$ is defined over $k$. Therefore, it suffices to prove the claim for an open dense subset. Since $U_{k^s}$ and $U_{\bar{k}}$ have the same underlying toplological space, by shrinking $U_{k^s}$, we may assume that $(U_{\bar{k}})_{\red}$ is irreducible and smooth over $\bar{k}$. So there exists a finite extension $l/k$ such that $(U_{l})_\red$ is irreducible and smooth over $l$. We may assume that $l/k$ is purely inseparable. Then $l^s=l\otimes_{k}k^s$ and $G_l=G_k$. Let $V$ denote $(U_{l})_\red$. Thus $V_{l^s}=V\times_{\Spec k} \Spec k^s$. The $k$-morphism $V\lra U$ induces a $G_k$-equivariant isomorphism
$$ H^1(U_{k^s}, \ZZ/\ell^n)\cong H^1(V_{l^s},\ZZ/\ell^n),$$
so it suffices to prove the claim for $V$. Thus we may assume that $U$ is smooth and geometrically connected over $k$.\\
By above arguments, we can always replace $k$ by a finite extension and shrink $U$. By de Jong's alteration theorem, we may assume that there is a finite flat morphism $f:V\lra U$ such that $V$ admits a smooth projective geometrically connected compactification over $k$. Since the kernel of 
$$H^1(U_{k^s},\ZZ/\ell^n)\lra H^1(V_{k^s},\ZZ/\ell^n)$$
is killed by the degree of $f$ (cf.\cite[Prop. 3.8.4]{CTS2}), it suffices to prove the claim for $V$. Thus we may assume that $U$ is an open subvariety of a smooth projective geometrically connected variety $X$ over $k$. Let $Y_i$ be irreducible components of $X-U$ of codimension $1$. Let $D_i$ be the regular locus of $Y_i$. By extending $k$ to a finite separable extension, we may assume that $D_i$ is geometrically irreducible. By purity, there is a canonical exact sequence

$$ 0\lra H^1(X_{k^s},\ZZ/\ell^n)\lra H^1(U_{k^s},\ZZ/\ell^n) \lra\bigoplus_{i}H^0(D_{i,k^s},\ZZ/\ell^n(-1))
$$
Taking $G_k$-invariants, we get an exact sequence
$$ 0\lra H^1(X_{k^s},\ZZ/\ell^n)^{G_k}\lra H^1(U_{k^s},\ZZ/\ell^n)^{G_k} \lra\bigoplus_{i}H^0(D_{i,k^s},\ZZ/\ell^n(-1))^{G_k}.
$$
It suffices to show that the size of the first and third group are bounded independent of $n$ and is equal to $1$ for all but finitely many $\ell$. 
Since $D_{i,k^s}$ is connected, we have
$$H^0(D_{i,k^s},\ZZ/\ell^n(-1))=\Hom(\mu_{\ell^n},\ZZ/\ell^n).$$
By the lemma below, the size of the third group is bounded independent of $n$ and is equal to $1$ for all but finitely many $\ell$.
Since
$$ H^1(X_{k^s},\ZZ/\ell^n)\cong \Pic_{X/K}[\ell^n](-1)$$
and
$$0\lra\Pic^0_{X/K,\red}[\ell^n]\lra \Pic_{X/K}[\ell^n]\lra \NS(X_{k^s})[\ell^n]$$
is exact, it suffices to show that the size of $(\Pic^0_{X/K,\red}[\ell^n](-1))^{G_k}$ is bounded independent of $n$ and is equal to $1$ for all but finitely many $\ell$.
Let $A$ be the dual of $\Pic^0_{X/K,\red}$. By the Weil pairing, 
$$\Pic^0_{X/K,\red}[\ell^n](-1)\cong \Hom(A[\ell^n],\ZZ/\ell^n).$$
Taking $G_k$-invariants, we get
$$(\Pic^0_{X/K,\red}[\ell^n](-1))^{G_k}\cong \Hom(A[\ell^n],\ZZ/\ell^n)^{G_k}.$$
By the lemma below, the claim follows.
\end{proof}
\begin{lem}
Let $A$ be an abelian variety over a finitely generated field $K$ of characteristic $p\geq 0$. Let $\ell\neq p$ be a prime.  Then the sizes of 

$$\Hom(A[\ell^n],\ZZ/\ell^n)^{G_K} \quad \mathrm{and} \quad  \Hom(\mu_{\ell^n},\ZZ/\ell^n)^{G_K} $$
are bounded independent of $n$. Moreover, for all but finitely many $\ell$, these two groups vanish for any $n$.
\end{lem}
\begin{proof}
We will only prove the claim for the first group, since the second one follows from the same arguments. Firstly, we assume that $K$ is finite, then we will use a specialization technique to reduce the general case to the finite field case.

If $K$ is finite, then we have
$$\Hom(A[\ell^n],\ZZ/\ell^n)^{G_K}=\Hom(A[\ell^n]_{G_K},\ZZ/\ell^n),$$
which has the same size as $A[\ell^n]^{G_{K}}$. Then the claim follows from the finiteness of $A(K)$. 

In general, choose an integral regular scheme $S$ of finite type over $\Spec \ZZ$ with function field $K$. By Shrinking $S$, we may assume $A/K$ extends to an abelian scheme $\SA/S$. Fix a closed point $s\in S$. For any $\ell\neq \Char(k(s))$, we can shrink $S$ such that $\ell$ is invertible on $S$ and $s\in S$. Then the \'etale sheaf $\SA[\ell^n]$ is a locally constant sheaf of $\ZZ/\ell^n$-module since $\SA[\ell^n]$ is finite \'etale over $S$. Thus
$$\SH om(\SA[\ell^n],\ZZ/\ell^n)$$
is also a locally constant sheaf of $\ZZ/\ell^n$-module and its stalk at the generic point can be identified with
$$ \Hom(A[\ell^n],\ZZ/\ell^n).$$
Set $\SF=\SH om(\SA[\ell^n],\ZZ/\ell^n)$, we have
$$\Hom(A[\ell^n],\ZZ/\ell^n)^{G_K}=H^0(S,\SF)\hookrightarrow H^0(s,\SF).
$$
Since $H^0(s,\SF)=\Hom(\SA_{s}[\ell^n],\ZZ/\ell^n)^{G_{k(s)}}$ and $k(s)$ is finite, thus the claim holds for $\ell\neq \Char(k(s))$. In the case $\Char(K)=0$, we can choose another closed point $s^\prime$ such that $\Char(k(s^\prime))\neq \Char(k(s))$, then by the same argument, the claim also holds for $\ell=\Char(k(s))$. This completes the proof.
\end{proof}

\subsection{The Tate-Shafarevich group}
In this section, first, we will study a geometric version of the Tate-Shafarevich group for an abelian variety over a function field with a base field $k$. Then, we will prove that the Galois fixed part of the geometric Tate-Shafarevich group is canonically isomorphic to Tate-Shafarevich group up to finite groups in the case that $k$ is a finite field. The idea is reducing the question to the relation between arithmetic Brauer groups and geometric Brauer groups which was studied in \cite{CTS1} (cf. \cite{Yua2}).
\begin{prop}\label{geosha}
Let $C$ be a smooth projective geometrically connected curve defined over a finitely generated field $k$ of characteristic $p\geq 0$. Let $K$ be the function field of $C$ and $A$ be an abelian variety over $K$. Denote $Kk^s$ by $K^\prime$.  Let $U\subseteq C$ be an open dense subscheme. Define 
$$\Sha_{U_{k^s}}(A):=\Ker(H^1(K^\prime,A)\lra \prod_{v\in |U_{k^s}|} H^1(K_v^{sh}, A))
$$
Then the natural map 
$$ 
\Sha_{C_{k^s}}(A)^{G_k}(\non p)\longrightarrow \Sha_{U_{k^s}}(A)^{G_{k}}(\non p)
$$
is injective and has a cokernel of finite exponent.
\end{prop}
\begin{proof}
By definitions, the injectivity is obvious. It suffices to show that the cokernel is of finite exponent. Since there exists an abelian variety $B/K$ such that $A\times B$ is isogenous to $\Pic^0_{X/K}$ for some smooth projective geometrically connected curve $X$ over $K$ (cf. \cite[p10]{CTS1}), it suffices to prove the claim for $A=\Pic^0_{X/K}$. Without loss of generality, we may replace $k$ by a finite extension $l/k$. This is obvious if $l/k$ is separable. For $l/k$ a purely inseparable extension of degree $p^n$, the quotient $A(K^sl^s)/A(K^s)$ is killed by $p^n$. Then $H^1(K^\prime,A) \lra H^1(Kl^s,A)$ has a kernel and a cokernel killed by $p^n$. It follows that $\Sha_{U_{k^s}}(A)\lra \Sha_{U_{l^s}}(A)$ has a kernel and a cokernel killed by some power of $p$. By resolution of singularity of surfaces (cf. \cite{Lip}), $X\lra \Spec K$ admits a proper regular model $\pi:\CX\longrightarrow C$. By extending $k$, we may assume that $\CX$ is smooth proper geometrically connected over $k$. Write $V$ for $\pi^{-1}(U)$. By the Leray spectral sequence
$$E^{p,q}_2=H^p(U_{k^s},R^q\pi_*\GG_m)\Rightarrow H^{p+q}(V_{k^s},\GG_m),
$$
we get a long exact sequence\\
$$H^2(U_{k^s},\mathbb{G}_m)\longrightarrow \Ker(H^2(V_{k^s},\mathbb{G}_m)
\longrightarrow H^0(U_{k^s},R^2\pi_*\mathbb{G}_m))$$
$$\longrightarrow H^1(U_{k^s},R^1\pi_*\mathbb{G}_m)\longrightarrow H^3(U_{k^s},\mathbb{G}_m).
$$
By \cite[Cor. 3.2 and Lem. 3.2.1]{Gro3}, $R^2\pi_*\GG_m=0$ and $H^i(U_{k^s},\GG_m)(\non p)=0$ for $i\geq 2$. Thus, we have a canonical isomoprhism
$$ H^2(V_{k^s},\GG_m)(\non p)\cong H^1(U_{k^s},R^1\pi_*\GG_m)(\non p).$$
Let $j:\Spec K^\prime \lra C_{k^s}$ be the generic point.
By the spectral sequence
$$H^p(U_{k^s},R^qj_*(j^*R^1\pi_*\GG_m))\Rightarrow H^{p+q}(K^\prime, \Pic(X_{K^s})),$$
we have
$$H^1(U_{k^s},j_*j^*R^1\pi_*\GG_m)=\Ker(H^1(K^\prime,\Pic(X_{K^s}))\lra \prod_{v\in |U_{k^s}|} H^1(K_v^{sh}, \Pic(X_{K^s}))).$$
Let $D$ be an effective divisor on $X$ with $\deg(D)>0$. Since the quotient $\Pic(X_{K^s})/(\Pic^0_{X/K}(K^s)\oplus \ZZ[D])$ is finite, the natural map induced by the inclusion $\Pic^0_{X/K}(K^s)\lra \Pic(X_{K^s})$
$$\Sha_{U_{k^s}}(\Pic^0_{X/K}) \lra H^1(U_{k^s},j_*j^*R^1\pi_*\GG_m)$$
has a kernel and a cokernel of finite exponent. Thus, the question is reduced to show that the natural map
$$ (H^1(C_{k^s},j_*j^*R^1\pi_*\GG_m)(\non p))^{G_k} \lra (H^1(U_{k^s},j_*j^*R^1\pi_*\GG_m)(\non p))^{G_k}$$
has a cokernel of finite exponent.
Without loss of generality, we may shrink $U$ such that $\pi$ is smooth on $\pi^{-1}(U)$. Then, by Lemma \ref{picard}, the natural map
$$R^1\pi_*\GG_m\lra j_*j^*R^1\pi_*\GG_m$$
is an isomorphism on $U$. It follows that
$$H^2(V_{k^s},\GG_m)(\non p)\cong H^1(U_{k^s},R^1\pi_*\GG_m)(\non p)$$
$$\cong H^1(U_{k^s},j_*j^*R^1\pi_*\GG_m)(\non p).
$$
Consider the following commutative diagram\\
\begin{displaymath}
\xymatrix{ \Br(\CX_{k^s})(\non p)\ar[r] \ar[d] &H^1(C_{k^s},j_*j^*R^1\pi_*\GG_m)(\non p)\ar[d]\\
	\Br(V_{k^s})(\non p)\ar[r]  &H^1(U_{k^s},j_*j^*R^1\pi_*\GG_m)(\non p) .}
\end{displaymath}
Since the map on the bottom is an isomorphism, it suffices to show that
the natural map
$$ \Br(\CX_{k^s})^{G_k}(\non p)\lra \Br(V_{k^s})^{G_k}(\non p) $$
has a finite cokernel. This follows from Proposition \ref{purity}.
\end{proof}

\begin{prop}\label{agsha}
Notations as in the above proposition. Assuming that $k$ is a finite field, define 
$$\Sha_{K^\prime}(A):=\Ker(H^1(K^\prime,A)\lra \prod_{v\in |C_{k^s}|} H^1(K_v^{sh}, A)).
$$
Then the natural map 
$$\Sha(A)(\non p)\longrightarrow \Sha_{K^\prime}(A)^{G_k}(\non p)
$$
has a kernel and a cokernel of finite exponent.
\end{prop}
\begin{proof}
We use the same arguments as in the proof of the previous lemma. It suffices to prove the claim for $A=\Pic^0_{Y/K}$ where $Y$ is a smooth projective geometrically connected curve over $k$. $Y$ admits a projective regular model $ \CY\longrightarrow C$. Then we have
$$ \Br(\mathcal{Y}) \cong \Sha(\Pic_{Y/K}^0) $$ 
and 
$$ \Br(\mathcal{Y}_{k^s}) \cong  \Sha_{K^\prime}(\Pic_{Y/K}^0) $$
up to finite groups. Thus the question is reduced to show that
$$
\Br(\mathcal{Y})\longrightarrow \Br(\mathcal{Y}_{k^s})^{G_k}
$$ 
has a finite kernel and a finite cokernel. This follows from \cite[Cor. 1.4]{Yua2}.
\end{proof}
\subsection{Cofiniteness of Brauer groups and Tate-Shafarevich groups}
Let $\ell$ be a prime number. Recall that a $\ell$-torsion abelian group $M$ is of \emph{cofinite type} if $M[\ell]$ is finite. This is also equivalent to that $M$ can be written as $(\QQ_\ell/\ZZ_\ell)^{r}\oplus M_0$ for some integer $r\geq 0$ and some finite group $M_0$. We also say that a torsion abelian group $M$ is of cofinite type if $M(\ell)$ is of cofinite type for all primes $\ell$. It is easy to see that a morphism $M\rightarrow N$ between abelian groups of cofinite type has a kernel and a cokernel of finite exponent if and only if the kernel and cokernel are finite.

In the following, we will show the cofinitness of geometric Tate-Shafarevich groups and geometric Brauer groups defined in previous sections.

\begin{lem}\label{cofsha}
The group $\Sha_{U_{k^s}}(A)(\ell)$ defined in the Proposition \ref{geosha} is of cofinite type for any $\ell\neq p$.
\end{lem}
\begin{proof}
We may shrink $U$ such that $A$ extends to an abelian scheme $\SA$ over $U$. By the exact sequence
$$0\lra\SA[\ell]\lra\SA\stackrel{\ell}{\lra}\SA\lra 0$$
we get a surjection
$$H^1(U_{k^s},\SA[\ell])\lra H^1(U_{k^s}, \SA)[\ell].$$
Since $H^1(U_{k^s},\SA[\ell])$ is finite, $H^1(U_{k^s}, \SA)[\ell]$ is also finite. Thus, it suffices to show
$$H^1(U_{k^s}, \SA)\cong\Sha_{U_{k^s}}(A). $$
This follows from  $\SA\cong j_*A$, since $\SA$ is a N\'eron model of $A$ over $U$, where $j:\Spec K\lra U$ is the generic point.

\end{proof}
\begin{lem}\label{cofbr}
Let $X$ be a smooth variety over a separable closed field $k$ of characteristic $p\geq 0$. Let $\ell\neq p$ be a prime. Then $\Br(X)(\ell)$
is of cofinite type.
\end{lem}
\begin{proof}
The Kummer exact sequence 
$$0\lra \mu_\ell \lra \GG_m\stackrel{\ell}{\lra} \GG_m \lra 0$$
induces a surjection
$$H^2(X,\mu_\ell)\lra \Br(X)[\ell].$$
Then the claim follows from the finiteness of $H^2(X,\mu_\ell)$.

\end{proof}
\section{Proof of the main theorem }
\subsection{The left exactness in Theorem \ref{thm1.2}}
In this section, we will prove the left exactness of the sequence in Theorem \ref{thm1.2} by following Grothendieck's arguments in \cite[\S 4]{Gro3}.
\begin{lem}\label{picard}
Let $U$ be an irreducible regular scheme of dimension $1$ with function field $K$. Let $\pi:\CX\lra U$ be a smooth proper morphism with a generic fiber geometrically connected over $K$. Let $j:\Spec K\lra U$ be the generic point of $U$. Then we have
\begin{itemize}
\item[(a)]
the natural map
$$R^1\pi_*\GG_m\lra j_*j^*R^1\pi_*\GG_m$$
is an isomorphism,
\item[(b)]
the natural map
$$R^2\pi_*\GG_m(\ell)\lra j_*j^*R^2\pi_*\GG_m(\ell)$$
is an isomorphism for any prime $\ell$ invertible on $U$.
\end{itemize}
\end{lem}
\begin{proof}
It suffices to show that the induced maps on stalks are isomorphism. Thus, we may assume that $U=\Spec R$ where $R$ is a strictly henselian DVR.

Let $X$ denote the generic fiber. Let $s \in U$ be the closed point. Then we have
 $$(R^1\pi_*\GG_m)_{\bar{s}}=\Pic(\CX) \quad \mathrm{and} \quad (j_*j^*R^1\pi_*\GG_m)_{\bar{s}}=\Pic_{X/K}(K).$$ 
Since $\CX_s$ admits a section $s\lra \CX_s$ and $\pi$ is smooth, the section can be extended to a section $U\lra \CX$. Thus $X(K)$ is not empty. So
$$\Pic_{X/K}(K)=\Pic(X).$$
Since $\CX$ is regular, the natural map
$$\Pic(\CX)\lra\Pic(X)$$
is surjective and has a kernel generated by vertical divisors. It suffices to show that $\CX_s$ is connected. This actually follows from $\pi_*\SO_{\CX}=\SO_U$ (cf. \cite[Chap. III, Cor. 11.3]{Har}). This proves $(a)$.

Let $I$ denote $G_K$. For $(b)$, the induced map on the stalk at $s$ is
$$\Br(\CX)(\ell)\lra \Br(X_{K^s})^I(\ell).
$$
Since $\pi$ is smooth and proper, we have 
$$ H^2(\CX,\mu_{\ell^\infty})\cong H^2(X_{K^s}, \mu_{\ell^\infty})= H^2(X_{K^s}, \mu_{\ell^\infty})^I.
$$
Consider\\
\begin{displaymath}
\xymatrix{
0\ar[r] & \Pic(\CX)\otimes \QQ_\ell /\ZZ_\ell\ar[r] \ar[d] & H^2(\CX , \mu_{\ell^\infty}) \ar[r] \ar[d] & \Br(\CX)(\ell) \ar[r]\ar[d] & 0  \\
0\ar[r] & \NS(X_{K^s})\otimes \QQ_\ell /\ZZ_\ell \ar[r]  & H^2(X_{K^s}, \mu_{\ell^\infty}) \ar[r]  & \Br(X_{K^s})(\ell) \ar[r] & 0 }
\end{displaymath}
Since $\NS(X_{K^s})\otimes \QQ_\ell /\ZZ_\ell$ is $I$-invariant and $\Pic(\CX)=\Pic(X)$. It suffices to show that
$$\Pic(X)\otimes \QQ_\ell /\ZZ_\ell \longrightarrow (\NS(X_{K^s})\otimes \QQ_\ell /\ZZ_\ell)^I
$$ 
is surjective. Write $\NS(X_{K^s})_{free}$ for $\NS(X_{K^s})/\NS(X_{K^s})_{\tor}$. The action of $I$ on $\NS(X_{K^s})$ factors through a finite quotient $I^\prime$. Consider the exact sequence
$$ 0\longrightarrow (\NS(X_{K^s})_{free})^{I^\prime} \otimes \ZZ_{\ell}\longrightarrow  \NS(X_{K^s})^{I^\prime}\otimes \QQ_{\ell}\longrightarrow  (\NS(X_{\bar{K}}) \otimes \QQ_{\ell}/\ZZ_\ell)^I$$
$$
\longrightarrow H^1(I^\prime,\NS(X_{K^s})_{free} \otimes \ZZ_{\ell}).
$$
$H^1(I^\prime,\NS(X_{K^s})_{free} \otimes \ZZ_{\ell})$ is killed by the order of $I^\prime$. Since $(\NS(X_{K^s}) \otimes \QQ_{\ell}/\ZZ_\ell)^I=\NS(X_{K^s}) \otimes \QQ_{\ell}/\ZZ_\ell$ is divisible, so the image of the last map is zero. Since 
$$\Pic(X)\otimes\QQ_\ell \longrightarrow \NS(X_{K^s})^I \otimes \QQ_{\ell}
$$
is surjective (cf. \cite[\S2.2]{Yua2}), the claim follows. By the Snake Lemma, the natural map
$$\Br(\CX)(\ell) \lra \Br(X_{K^s})^{I}(\ell)$$
is an isomorphism.

\end{proof}

\begin{lem}\label{exact}
Let $U$ be a smooth geometrically connected curve over a field $k$ of characteristic $p\geq 0$ with function field $K$. Let $\pi:\CX\lra U$ be a smooth proper morphism with a generic fiber $X$ geometrically connected over $K$. Let $K^\prime$ denote $Kk^s$.
Define
 $$\Sha_{U_{k^s}}(\Pic_{X/K}):=\Ker(H^1(K^\prime,\Pic(X_{K^s}))\lra \prod_{v\in |U_{k^s}|} H^1(K_v^{sh}, \Pic(X_{K^s}))).
$$
Then we have a canonical exact sequence
$$0\ra(\Sha_{U_{k^s}}(\Pic_{X/K}))^{G_k}(\non p)\ra \Br(\CX_{k^s})^{G_k}(\non p)\ra \Br(X_{K^s})^{G_K}(\non p).$$
Moreover, the natural map
$$(\Sha_{U_{k^s}}(\Pic^0_{X/K,\red}))^{G_k}\lra (\Sha_{U_{k^s}}(\Pic_{X/K}))^{G_k}
$$
has a kernel and cokernel of finite exponent.

\end{lem}
\begin{proof}
We may assume that $k=k^s$. 
By the Leray spectral sequence
$$E^{p,q}_2=H^p(U,R^q\pi_*\GG_m)\Rightarrow H^{p+q}(\CX,\GG_m),
$$
we get a long exact sequence\\
$$H^2(U,\mathbb{G}_m)\longrightarrow \Ker(H^2(\CX,\mathbb{G}_m)
\longrightarrow H^0(U, R^2\pi_*\mathbb{G}_m))$$
$$\longrightarrow H^1(U,R^1\pi_*\mathbb{G}_m)\longrightarrow H^3(U,\mathbb{G}_m)
$$
By \cite[Lem. 3.2.1]{Gro3}, $H^i(U,\GG_m)(\non p)=0$ for $i\geq 2$. Thus, we get a canonical exact sequence
$$ 0\ra H^1(U,R^1\pi_*\GG_m)(\non p) \ra H^2(\CX,\GG_m) (\non p)\ra H^0(U,R^2\pi_*\GG_m)(\non p).$$
Let $j:\Spec K \lra U$ be the generic point. By Lemma \ref{picard}, we have  canonical isomorphisms
$$H^1(U,R^1\pi_*\GG_m)(\non p)\cong H^1(U,j_*j^*R^1\pi_*\GG_m)(\non p)$$ 
and
$$H^0(U,R^2\pi_*\GG_m)(\non p)\cong H^0(U,j_*j^*R^2\pi_*\GG_m)(\non p).
$$
Since $j^*R^i\pi_*\GG_m$ corresponds to the $G_K$-module $H^i(X_{K^s},\GG_m)$, we have
$$H^1(U,j_*j^*R^1\pi_*\GG_m)=\Sha_{U}(\Pic_{X/K})\quad \mathrm{and} \quad 
H^0(U,j_*j^*R^2\pi_*\GG_m)= \Br(X_{K^s})^{G_K}.
$$
This proves the first claim.

For the second claim, consider the exact sequence
$$0\lra \Pic^0(X_{K^s})\lra \Pic(X_{K^s})\lra \NS(X_{K^s})\lra 0.$$
Since $\NS(X_{K^s})$ is finitely generated, there exists a finite Galois extension $L/K$ such that
$\Pic(X_{L})\lra\NS(X_{K^s})$
is surjective. Taking cohomoloy, we get a long exact sequence
$$0\lra H^0(K,\Pic^0(X_{K^s}))\lra H^0(K,\Pic(X_{K^s}))\lra H^0(K,\NS(X_{K^s}))$$
$$\stackrel{a_K}{\lra} H^1(K,\Pic^0(X_{K^s}))\lra H^1(K,\Pic(X_{K^s}))\lra H^1(K,\NS(X_{K^s})).$$
We have a similar long exact sequence for $H^i(L,-)$. Since $\Pic(X_L)\rightarrow \NS(X_{K^s})$ is surjective, we have $a_L=0$ and 
$$H^1(L,\NS(X_{K^s}))=\Hom(G_L,\NS(X_{K^s}))=\Hom(G_L,\NS(X_{K^s})_{\tor}).$$
$a_L=0$ implies that the image of $a_K$ is contained in
$$ \Ker(H^1(K,\Pic^0(X_{K^s}))\lra H^1(L,\Pic^0(X_{K^s})))$$
which is killed by $[L:K]$. Similarly, one can show that $H^1(K,\NS(X_{K^s}))$ is killed by $[L:K]|\NS(X_{K^s})_{\tor}|$. Therefore, the kernel and cokernel of 
$$H^1(K,\Pic^0(X_{K^s}))\lra H^1(K,\Pic(X_{K^s}))$$
are killed by $[L:K]|\NS(X_{K^s})_{\tor}|$. The claim also holds for 

$$H^1(K_v^{sh},\Pic^0(X_{K^s}))\lra H^1(K^{sh}_v,\Pic(X_{K^s})).$$
By diagram chasings, the kernel and cokernel of
$$\Sha_{U}(\Pic^0_{X/K,\red})\lra \Sha_{U}(\Pic_{X/K})
$$
is of finite exponent. This completes the proof.
\end{proof}

\subsection{The pull-back trick}
In this section, we will introduce a pull-back trick initially developed by Colliot-Thél\`ene and Skorobogatov in \cite{CTS1} which allows us to reduce the question to cases of relative dimension $1$.

Let $U$ be a regular integral excellent scheme of dimension $1$ with function field $K$. Let $\pi:\CX\lra U$ be a smooth projective morphism with the generic fiber $X$ geometrically connected over $K$. The Leray spectral sequence 
$$E_2^{p,q}=H^p(U,R^q\pi_*\GG_m)\Rightarrow H^{p+q}(\CX,\GG_m)$$
induces canonical maps
$$d^{1,1}_2:E_2^{1,1}\lra E_2^{3,0},$$
$$d_3^{0,2}:E_3^{0,2}\lra E_3^{3,0},$$
$$d_2^{0,2}:E_2^{0,2}\lra E_2^{2,1}.$$
\begin{lem}\label{keylem}
Assuming that $X(K)$ is not empty and the natural map $\Pic(X)\rightarrow\NS(X_{K^s})$ is surjective, then the canonical maps $d^{1,1}_2$ and $d_3^{0,2}$ vanish and $E_3^{3,0}=E_2^{3,0}$. Moreover, there exists an open dense subscheme $V\subseteq U$ such that the canonical map $d_2^{0,2}$ has an image of finite exponent when replacing $U$ by $V$. As a result, the natural map
$$H^2(\pi^{-1}(V),\GG_m)\lra H^0(V,R^2\pi_*\GG_m)$$
has a cokernel of finite exponent.
\end{lem}
\begin{proof}
Let $s\in X(K)$.  Since $\pi$ is proper, it extends to a section $s:U\lra \CX$. Let $\tilde{E}_2^{p,q}$ denote the Leray spectral sequence for the identity map $U\lra U$. Then $s$ induces a commutative diagram
\begin{displaymath}
\xymatrix{E_2^{1,1}\ar[r]^{d_2^{1,1}}\ar[d]& E_2^{3,0}\ar[d] \\
\tilde{E}_2^{1,1}\ar[r]& \tilde{E}_2^{3,0}}
\end{displaymath}
The second column is an isomorphism since $E_2^{3,0}=H^3(U,\GG_m)=\tilde{E}_2^{3,0}$. Since $\tilde{E}_2^{1,1}=0$, thus $d_2^{1,1}=0$. By definition, $d_2^{3,0}=0$, it follows that $E_3^{3,0}=E_2^{3,0}$. By the same arguments, we have $d_3^{0,2}=0$.

For the proof of the second claim, we will use the pullback method for finitely many $U$-morphisms $\CY_i\lra\CX$ where $\CY_i$ is of relative dimension $1$ over $U$. Let $^iE^{p,q}_2$ denote the Leray spectral sequence for $\pi_i:\CY_i\lra U$. There is a commutative diagram
\begin{displaymath}
\xymatrix{E_2^{0,2} \ar[r]^{d_2^{0,2}}\ar[d]& E_2^{2,1}\ar[d] \\
\oplus_i \,^iE_2^{0,2}\ar[r]& \oplus_i \,^iE_2^{2,1}}
\end{displaymath}
By shrinking $U$, we can assume that $\pi_i$ is smooth and projective for all $i$. By Artin's theorem \cite[Cor. 3.2]{Gro3}, $R^2\pi_{i,*}\GG_m=0$. Thus, to show that $d_2^{0,2}$ has an image of finite exponent, it suffices to show that the second column has a kernel of finite exponent. The second column is the map
$$ H^2(U,R^1\pi_*\GG_m)\lra \bigoplus_iH^2(U,R^1\pi_{i,*}\GG_m).$$
By Lemma \ref{picard}$(a)$, we have a canonical isomorphism
$$R^1\pi_{i,*}\GG_m\cong j_*j^*R^1\pi_{i,*}\GG_m,$$
where $j:\Spec K\lra U$ is the generic point. It suffices to show that there exists an \'etale sheaf $\SF$ on $\Spec K$ and a morphism $\SF\lra \oplus_ij^*R^1\pi_{i,*}\GG_m$ such that the induced map
\begin{equation}\label{splitness}
j^*R^1\pi_{*}\GG_m\oplus\SF \lra \bigoplus_ij^*R^1\pi_{i,*}\GG_m
\end{equation}
has a kernel and a cokernel killed by some positive integer. This will imply that the induced map
$$H^2(U,j_*j^*R^1\pi_{*}\GG_m)\oplus H^2(U,j_*\SF) \lra \bigoplus_i H^2(U,j_*j^*R^1\pi_{i,*}\GG_m)
$$
has a kernel and a cokernel of finite exponent. Since $j^*R^1\pi_*\GG_m$ corresponds to the $G_K$-module $\Pic(X_{K^s})$, (\ref{splitness}) can be interpreted as that there is a $G_K$-module $M$ and a $G_K$-equivariant map $M\lra \oplus_{i}\Pic(Y_{i,K^s})$ such that
$$\Pic(X_{K^s})\oplus M\lra \bigoplus_{i}\Pic(Y_{i,K^s})$$
has a kernel and a cokernel of finite exponent. In \cite[\S 3.3, p17]{Yua2},
Yuan proved that there exist smooth projective curves $Y_i$ in $X$ which are obtained by taking intersection of very ample divisors and an abelian subvariety $A$ of $\prod_i\Pic^0_{Y_i/K}$ such that the induced morphism
$$\Pic(X_{K^s})\times A(K^s)\lra \prod_i\Pic(Y_{i,K^s})$$
has a finite kernel and a finite cokernel. Taking $\CY_i$ to be the Zariski closure of $Y_i$ in $\CX$, then shrinking $U$, we get smooth and proper morphisms $\pi:\CY_i\lra U$. This proves the second claim.

For the last claim, consider the following canonical exact sequences induced by the Leray spectral sequence,
$$E^2\lra E_{4}^{0,2}\lra 0,$$
$$0\lra E_{4}^{0,2}\lra E_{3}^{0,2}\stackrel{d_3^{0,2}}{\lra} E_{3}^{3,0},$$
$$0\lra E_{3}^{0,2}\lra E_{2}^{0,2}\stackrel{d_2^{0,2}}{\lra} E_{2}^{2,1}.$$
Since $d_3^{0,2}$ vanishes and $d_2^{0,2}$ has an image of finite exponent, it follows that
$$E^2\lra E_2^{0,2}$$
has a cokernel of finite exponent. This completes the proof.

\end{proof}
\subsection{Proof of Theorem \ref{thm1.2}}
Now we prove Theorem \ref{thm1.2}. Let $U$ be an open dense subscheme of $C$ such that $\pi$ is smooth and projective over $U$. By Lemma \ref{exact}, the natural map
$$(\Sha_{U_{k^s}}(\Pic^0_{X/K,\red}))^{G_k}(\non p)\rightarrow \Ker(\Br(\pi^{-1}(U_{k^s}))^{G_k}(\non p)\rightarrow\Br(X_{K^s})^{G_K}(\non p))$$
has a kernel of finite exponent. By Lemma \ref{cofsha} and \ref{cofbr}, all groups here are of cofinite type, thus the kernel is actually finite. By Proposition \ref{purity} and \ref{geosha}, it suffices to show that the natural map
$$\Br(\pi^{-1}(U_{k^s}))^{G_k}(\non p)\lra \Br(X_{K^s})^{G_K}(\non p)$$
has a cokernel of finite exponent.

Firstly, we will prove this under the assumption that $X(K)$ is not empty and the natural map $\Pic(X)\rightarrow\NS(X_{K^s})$ is surjective. We still use $\CX$ to denote $\pi^{-1}(U)$. By Proposition \ref{purity}, without loss of generality, we can always replace $U$ by an open dense subset. By Lemma \ref{picard}$(b)$, we have
$$H^0(U, R^2\pi_*\GG_m)(\non p)\cong H^0(U, j_*j^*R^2\pi_*\GG_m)(\non p)=\Br(X_{K^s})^{G_K}(\non p),$$
so it suffices to show that the natural map induced by the Leray spectral sequence for $\pi$
$$ H^2(\CX,\GG_m)\lra H^0(U, R^2\pi_*\GG_m)$$
has a cokernel of finite exponent. By shrinking $U$, the claim follows from Lemma \ref{keylem}.

Secondly, we will show that the question can be reduced to the case that $X(K)$ is not empty and the natural map $\Pic(X)\rightarrow\NS(X_{K^s})$ is surjective. There exists a finite Galois extenion $L/K$ such that $X(L)$ is not empty and the natural map $\Pic(X_L)\rightarrow\NS(X_{K^s})$ is surjective.  Let $W$ be a smooth curve over $k$ with function field $L$. By shrinking $U$ and $W$, the map $\Spec L \lra \Spec K$ extends to a finite \'etale Galois covering map $W\lra U$. Let $\CX_W \lra W$ be the base change of $\CX\lra U$ to $W$. By the above arguments, the natural map
$$ \Br(\CX_W)(\non p)\lra\Br(X_{K^s})^{G_L}(\non p)$$
has a cokernel of finite exponent. Let $G$ denote $\Gal(L/K)$. The above map is compatible with the $G$-action. Taking $G$-invariant, we have that
$$\Br(\CX_W)^{G}(\non p)\lra\Br(X_{K^s})^{G_K}(\non p)$$
has a cokernel of finite exponent. Then the question is reduced to show that the natural map
$$\Br(\CX)\lra\Br(\CX_W)^{G}$$
has a cokernel of finite exponent. Consider the spectral sequence
$$H^p(G,H^q(\CX_W,\GG_m))\Rightarrow H^{p+q}(\CX,\GG_m).$$
Since $H^p(G,-)$ is killed by the order of $G$, by similar arguments as in Lemma \ref{keylem}, we can conclude that the cokernel of
$$H^2(\CX,\GG_m)\lra H^2(\CX_{W},\GG_m)^{G}$$
is of finite exponent. This completes the proof of the theorem.
\section{Applications}
\subsection{Proof of Theorem \ref{finitebrauer}}
Firstly, note that in the case that $k$ is finite, by \cite[Thm. 2.3]{Lic}, the finiteness of $\Br(\CX_{k^s})^{G_k}(\non p)$ is equivalent to the Tate conjecture for divisors on $\CX$.

We will prove the finiteness of $\Br(X_{K^s})^{G_K}(\non p)$ by the induction on the transcendence degree of $K/k$. If $K/k$ is of transcendence degree $0$, a smooth projective variety over $K$ can be view as a smooth projective variety over $k$. So the claim is true by assumptions.

Assume  that the claim is true for all extensions $l/k$ of transcendence degree $n$. Let $K/k$ be a finitely generated extension of transcendence degree $n+1$. Let $X/K$ be a smooth projective connected variety over $K$, by extending $K$, we may assume that $X/K$ is geometrically connected. $K$ can be regarded as the function field of a smooth projective geometrically connected curve $C$ over a field $l$ of transcendence degree $n$ over $k$. $X\lra \Spec K$ can be extended to a smooth morphism $\pi:\CX\lra C$. By Theorem \ref{thm1.2}, 
$$\Br(\CX_{l^s})^{G_k}(\non p)\lra \Br(\CX_{K^s})^{G_K}(\non p)
$$
has a finite cokernel. It suffices to show that $\Br(\CX_{l^s})^{G_l}(\non p)
$ is finite. By Proposition \ref{purity}, we may shrink $\CX$ such that there is an alteration $f:\CX^{\prime}\lra \CX$ such that $\CX^{\prime}$ admits a smooth projective compactification over a finite extension of $l$. Replacing $l$ by a finite extension will not change the question, so we may assume that $\CX^{\prime}$ admits a smooth projective compactification over $l$. By shrinking $\CX$, we may assume that $f$ is finite flat. Then the kernel of 
$$\Br(\CX_{l^s})(\non p)\lra\Br(\CX^\prime_{l^s})(\non p)$$
is killed by the degree of $f$ (cf. \cite[Prop. 3.8.4]{CTS2}) and therefore is finite. Since $\Br(\CX^\prime_{l^s})^{G_l}(\non p)$ is finite by induction, so $\Br(\CX_{l^s})^{G_l}(\non p)$ is finite. This completes the proof.

\subsection{Proof of Theorem \ref{reduction}}
Let $\ell\neq p$ be a prime, the the claim in Theorem \ref{finitebrauer} also holds when replacing the prime to $p$ part by the $\ell$ primary part in the statement. So by the same argument as in the proof of Theorem \ref{finitebrauer}, Theorem \ref{reduction} is true.

\subsection{Brauer groups of integral models of abelian varieties and K3 surfaces}
\begin{prop}
Let $C$ be a smooth projective geometrically connected curve over a finite field $k$ of characteristic $p>2$. Let $\pi:\CX \lra C$ be a projective flat morphism. Let $X$ denote the generic fiber of $\pi$. Assuming that
$\CX$ is regular and $X_{K^s}$ is an abelian variety over $K^s$, then there is an isomorphism up to finite groups
$$ \Sha(\Pic_{X/K,\red}^{0})(\non p)\cong \Br(\CX)(\non p).
$$

\end{prop}
\begin{proof}
By \cite[Thm. 1.1]{SZ1}, $\Br(X_{K^s})^{G_K}(\non p)$ is finite, the claim follows directly from Corollary \ref{mainthm}.
\end{proof}

\begin{prop}
Let $C$ be a smooth projective geometrically connected curve over a finite field $k$ of characteristic $p$. Let $\pi:\CX \lra C$ be a projective flat morphism. Assuming that $\CX$ is regular and the generic fiber $X$ of $\pi$ is a smooth projective geometrically connected surface over $K$ with $H^1(X_{K^s},\QQ_\ell)=0$ for some prime $\ell\neq p$, then the natural map
$$\Br(\CX)(\non p)\lra \Br(X_{K^s})^{G_K}(\non p)
$$
has a finite kernel and cokernel. As a result, $\Br(X_{K^s})^{G_K}(\non p)$ is finite if and only if $\Br(X_{K^s})^{G_K}(\ell)$ is finite for some prime $\ell\neq p$.
\end{prop}
\begin{proof}
$\Pic^0_{X/K,\red}=0$ since $H^1(X_{K^s},\QQ_\ell)=0$. It follows that
$$\Sha(\Pic_{X/K,\red}^{0})=0.$$
Then the claim follows from Corollary \ref{mainthm}.
\end{proof}
\begin{cor}
Notations as above, if $X/K$ is a $K3$ surface and $p>2$, then 
$\Br(\CX)$ is finite.
\end{cor}
\begin{proof}
Since Tate conjecture for $X$ is known by work of lots of people (cf. Conjecture \ref{Tateconj} ), so $\Br(X_{K^s})^{G_K}(\ell)$ is finite for some prime $\ell\neq p$. By the above proposition, $\Br(\CX)(\ell)$ is also finite. Then by \cite[Thm. 0.4]{Mil2}, $\Br(\CX)$ is finite.
\end{proof}
\begin{prop}
Let $X$ be a $K3$ surface over a global field $K$ of characteristic $p>2$.   Then  $\Br(X_{K^s})^{G_K}(\non p)$ is finite.
\end{prop}
\begin{proof} 
By previous results, it suffices to show that $X$ admits a projective regular model. By resolution of singularity of threefolds (cf. \cite{CP}), $X$ does admit a projective regular model.
\end{proof}
\section* {Acknowledgements}
I would like to thank my advisor Xinyi Yuan for suggesting this question to me and for telling me the pull-back trick. I also want to thank Jean-Louis Colliot-Thélène, Alexei N. Skorobogatov, Cristian D. Gonzalez-Aviles, Timo Keller, Dino Lorenzini, Thomas Geisser and the anonymous referees for very helpful comments and suggestions.


   \Addresses
\end{document}